\documentclass[11pt]{amsart}
\usepackage{amssymb,amsmath,amsthm,mathrsfs,multirow,enumerate}

\usepackage{graphicx} 
\usepackage[all]{xy}

\numberwithin{equation}{section} 

\theoremstyle{plain}
\newtheorem{thm}{Theorem}[section]

\newtheorem{lemma}[thm]{Lemma}
\newtheorem{prop}[thm]{Proposition}

\newtheorem*{nota}{Notation}

\theoremstyle{definition}

\newtheorem{ex}{Example}

\newcommand{\diag}{\mathop{\mathrm{diag}}}

\newcommand{\bbm}{\begin{bmatrix}}
\newcommand{\ebm}{\end{bmatrix}}

\begin{document}
\title{Cogredient Standard Forms of Orthogonal Matrices over Finite Local Rings of Odd Characteristic} 

\author{Yotsanan Meemark and Songpon Sriwongsa}

\address{Yotsanan Meemark\\ Department of Mathematics and Computer Science \\ Faculty of Science
\\ Chulalongkorn University\\ Bangkok, 10330 THAILAND}
\email{\tt yotsanan.m@chula.ac.th}

\address{Songpon Sriwongsa\\ Department of Mathematics and Computer Science\\ Faculty of Science
\\ Chulalongkorn University\\ Bangkok, 10330 THAILAND}
\email{\tt songpon\_sriwongsa@hotmail.com}


\keywords{Cogredient, Local rings, Orthogonal spaces.}

\subjclass[2000]{Primary: 05C25; Secondary: 05C60}

\begin{abstract}
In this work, we present a cogredient standard form of an orthogonal space over a finite local ring of odd characteristic.
\end{abstract}

\maketitle

\section{Units and the square mapping}
A {\it local ring} is a commutative ring which has a unique maximal ideal. For a local ring $R$,  we denote its unit group by $R^\times$ and it follows from Proposition 1.2.11 of \cite{B02} its unique maximal ideal $M = R \setminus R^\times$ consists of all non-unit elements. We also call the field $R/M$, the {\it residue field of }$R$. 

\begin{ex}
	If $p$ is a prime, then $\mathbb{Z}_{p^n}$, $n \in \mathbb{N}$, is a local ring with maximal ideal $p\mathbb{Z}_{p^n}$ and residue field $\mathbb{Z}_{p^n}/p\mathbb{Z}_{p^n}$ isomorphic to $\mathbb{Z}_p$. Moreover, every field is a local ring with maximal ideal $\{ 0\}$.
\end{ex}

Recall a common theorem about local rings that:

\begin{thm}
	Let $R$ be a local ring with unique maximal ideal $M$. Then $1 + m$ is a unit of $R$ for all $m \in M$. Furthermore, $u + m$ a unit in $R$ for all $m \in M$ and $u \in R^{\times}$.
	
\end{thm}
\begin{proof}
	Suppose that $1+m$ is not a unit. Since $R$ is local, $1+m \in M$. Hence, $1$ must be in $M$, which is a contradiction. Finally, we note that $u+m = u(1+u^{-1}m)$ is a unit in $R$.
\end{proof}

Let $R$ be a finite local ring of odd characteristic with unique maximal ideal $M$ and residue field~$k$. Then $R$ is of order an odd prime power, and so is $M$. From Theorem XVIII. 2 of \cite{M74} we have that the unit group of $R$, denoted by $R^\times$, is isomorphic to $(1 + M)\times k^{\times}$. Consider the exact sequence of groups 
\[
1 \longrightarrow K_R \longrightarrow R^\times \longrightarrow (R^\times)^2 \longrightarrow 1
\]
where $\theta : a \longmapsto a^2$ is the square mapping on $R^\times$ with kernel $K_R = \{a \in R^\times: a^2 = 1 \}$ and $(R^\times)^2 = \{a^2 : a \in R^\times \}$. Note that $K_R$ consists of the identity and all elements of order two in $R^\times$. Since $R$ is of odd characteristic and $k^\times$ is cyclic, $K_R = \{\pm 1 \}$. Hence, $[ R^\times : (R^\times)^2 ] = |K_R| = 2 $.

\begin{prop}\label{indexUR}
	Let $R$ be a finite local ring of odd characteristic with unique maximal ideal $M$ and residue field~$k$.
	\begin{enumerate}[(1)]
		\item  
		The image $(R^\times)^2$ is a subgroup of $R^\times$ with index $\left[ R^\times : (R^\times)^2 \right] = 2$.
		\item 
		For $z \in R^\times \setminus (R^\times)^2$, we have $R^\times \setminus (R^\times)^2 = z(R^\times)^2$ and $|(R^\times)^2| = |z(R^\times)^2| = (1/2)|R^\times|$.
		\item
		For $u \in R^\times$ and $a \in M$, there exists $c \in R^\times$ such that $c^2 (u + a) = u$.
		\item
		If $-1 \notin (R^\times)^2$ and $u \in R^\times$, then $1 + u^2 \in R^\times$.
		\item If $-1 \notin (R^\times)^2$ and $z \in R^\times \setminus (R^\times)^2$, then there exist $x, y \in R^\times$ such that $z = (1+x^2)y^2$.
	\end{enumerate}
	\begin{proof}
		We have proved (1) in the above discussion and (2) follows from (1). Let $u \in R^\times$ and $a \in M$. Then $u^{-1}(u+a) = 1 + u^{-1}a \in 1 + M$, so $(u^{-1}(u + a))^{| 1 + M | + 1} = u^{-1}(u + a)$. Since $|1+M| = |M|$ is odd, $u^{-1}(u+a) = (c^{-1})^2$ for some $c \in R^\times$. Thus, $c^2(u+a) = u$ which proves (3).
		
		For (4), assume that $-1 \notin (R^\times)^2$ and let $u \in R^\times$. Suppose that $1+u^2 = x \in M$. Then $u^2 = -(1-x)$. Since $|M|$ is odd and $1-x \in 1+M$, $(u^{|M|})^2 = (-(1-x))^{|M|} = (-1)^{|M|}(1-x)^{|M|} = (-1)(1) = -1$, which contradicts $-1$ is non-square. Hence, $1+u^2 \in R^\times$.
		
		Finally, we observe that $|1 + (R^\times)^2| = |(R^\times)^2|$ is finite. If $1 + (R^\times)^2 \subseteq  (R^\times)^2$, then they must be equal, so there exists $b \in (R^\times)^2$ such that $1 + b = 1$, which forces $b = 0$, a contradiction. Hence, there exists an $x \in R^\times$ such that $1 + x^2 \notin (R^\times)^2$. By (4), $1 + x^2 \in R^\times$. Therefore, for a non-square unit $z$, we have $R^\times$ is a disjoint union of cosets $(R^\times)^2$ and $z(R^\times)^2$, so $1 + x^2 = z(y^{-1})^2$ for some $y \in R^\times$ as desired.
	\end{proof}
\end{prop}

In what follows, we shall apply the above proposition to obtain a nice cogredient standard form of an orthogonal space over a finite local ring of odd characteristic. This work generalizes the results over a Galois ring studied in \cite{C12}.

\section{Cogredient standard forms of orthogonal spaces}

Throughout this section, we let $R$ be a finite local ring of odd characteristic.

\begin{nota}
	For any $l \times n$ matrix $A$ and $q \times r$ matrix $B$ over $R$, we write 
	\[A \oplus B := 
	\begin{pmatrix}
	A & 0\\
	0 & B
	\end{pmatrix}\]
	which is an  $( l +q ) \times (n+r)$ matrix over $R$.
\end{nota}

For any matrices $S_1, S_2 \in M_n(R)$, if there exists an invertible matrix $P$ such that $PS_1P^T = S_2$, we say that $S_1$ is \textit{cogredient} to $S_2$ over $R$ and we write $S_1 \approx S_2$. Note that $S \approx c^2 S$ for all $c \in R^\times$. The next lemma is a key for our structure theorem. 

\begin{lemma}\label{zI}
	For a positive integer $\nu$ and $z \in R^\times \setminus (R^\times)^2, zI_{2\nu}$ is cogredient to $I_{2\nu}$.
	\begin{proof}
		If $-1 = u^2$ for some $u \in R^\times$, we may choose  $P = 2^{-1}
		\begin{pmatrix}
		(1+z) & u^{-1}(1-z) \\
		u(1-z) & (1+z)
		\end{pmatrix}$ whose determinant is $z \in R^\times$. Note that our $R$ of odd characteristic, so $2$ is a unit. Hence, $P$ is invertible and $PP^T = z I_2$. 
		Next, we assume that $-1$ is non-square. Then, by Proposition \ref{indexUR} (5), $z = (1 + x^2)y^2$ for some units $x$ and $y$ in $R^\times$. Choose  $Q = 
		\begin{pmatrix}
		xy & y \\
		-y & xy
		\end{pmatrix}$. Then $\det Q = (1 + x^2)y^2 = z \in R^\times$, so $Q$ is invertible and $QQ^T = 
		\begin{pmatrix}
		(1 + x^2)y^2 & 0\\
		0 & (1 + x^2)y^2
		\end{pmatrix} = zI_2$. 
		Therefore, $zI_{2\nu} = \overbrace{zI_2 \oplus \cdots \oplus zI_2}^{\nu^, s}$ is cogredient to $I_{2\nu} = \overbrace{I_2 \oplus \cdots \oplus I_2}^{\nu^,s}$.
	\end{proof}
\end{lemma}

Let $R$ be a local ring. Let $V$ be a free $R$-module of rank $n$, where $n \geq 2$. Assume that we have a function $\beta : V \times V \rightarrow R$ which is $R$-bilinear, symmetric and the $R$-module morphism from $V$ to $V*=\hom_R(V, R)$ given by $\vec{x} \mapsto \beta(\cdot, \vec{x})$ is an isomorphism. For $\vec{x} \in V$, we call $\beta(\vec{x}, \vec{x})$ the {\it norm of} $\vec{x}$. The pair $(V, \beta)$ is called an {\it orthogonal space}. Moreover, if $\beta = \{ \vec{b}_1,\ldots, \vec{b}_n \}$ is a basis of $V$, then the associated matrix $[\beta]_\mathcal{B} = [\beta(\vec{b}_i, \vec{b}_j)]_{n \times n}$. We say that $\mathcal{B}$ is an orthogonal basis if $\beta(\vec{b}_i, \vec{b}_i) = u_i \in R^{\times}$ for all $i$ and $\vec{b}_i, \vec{b}_j = 0$ for $i \neq j$. 

McDonald and Hershberger \cite{MH78} proved the following theorem.

\begin{thm}[Theorem 3.2 of \cite{MH78}]
	Let $(V,\beta)$ be an orthogonal space of rank $n \ge 2$. Then $(V, \beta)$ processes an orthogonal basis $\mathcal{C}$ so that $[\beta]_{\mathcal{C}}$ is a diagonal matrix whose entries on the diagonal are units.
\end{thm}

Let $(V,\beta)$ be an orthogonal space of rank $n \ge 2$. Let $\mathcal{C}$ be an orthogonal basis of $V$ such that $[\beta]_{\mathcal{C}}$ is a diagonal matrix whose entries on the diagonal are units.
From $[\beta]_\mathcal{C}= \diag (u_1, \ldots, u_n)$ and $u_i$ are units for all $i$. Assume that $u_1, \ldots, u_r$ are squares and $u_{r+1}, \ldots, u_n$ are non-squares. Since $R^\times$ is a disjoint union of the cosets $(R^\times)^2$ and $z(R^\times)^2$ for some non-square unit $z$, we have $u_i = w_i^2$ for some $w_i \in R^\times$, $i = 1, \ldots, r$ and  $u_j = zw_j^2$ for some $w_j \in R^\times$, $j = r+1, \ldots, n$. Thus, $[\beta]_\mathcal{C}= \diag (u_1, \ldots, u_r) \oplus z \diag (w_{r+1}, \ldots, w_n)$ which is cogredient to $I_r \oplus zI_{n-r}$. If $n-r$ is even, Lemma \ref{zI} implies that $[\beta]_\mathcal{C}$ is cogredient to $I_n$. If $n - r$ is odd, then $n - r - 1$ is even and so $[\beta]_\mathcal{C}$ is cogredient to $I_{n-1} \oplus (z)$ by the same lemma. Note that $I_n$ and $I_{n - 1} \oplus (z)$ are not cogredient since $z$ is non-square. We record this result in the next theorem.

\begin{thm}\label{2type}
	Let $z$ be a non-square unit in $R$. Then $[\beta]_\mathcal{C}$ is cogredient to either $I_n$ or $I_{n-1} \oplus (z)$. 
\end{thm}

The next lemma follows by a simple calculation.

\begin{lemma}\label{cH}
	Let $z$ be a non-square unit in $R$ and  and $\nu$ a positive integer. Write $H_{2\nu}=
	\begin{pmatrix}
	0 & I_\nu \\
	I_\nu & 0
	\end{pmatrix}$.
	\begin{enumerate}[(1)]
		\item 
		If $-1 \in (R^\times)^2$, then $I_\nu$ is cogredient to $H_{2\nu}$ and $
		\begin{pmatrix}
		1 & 0 \\
		0 & z
		\end{pmatrix}
		\approx
		\begin{pmatrix}
		1 & 0 \\
		0 & -z
		\end{pmatrix}.
		$
		\item
		If $-1 \notin (R^\times)^2$, then $I_\nu \oplus zI_\nu$ is cogredient to $H_{2\nu}$ and $I_2 \approx 
		\begin{pmatrix}
		1 & 0 \\
		0 & -z
		\end{pmatrix}.
		$
	\end{enumerate}
	\begin{proof} 
		First we observe that if $-1 = u^2$ for some unit $u$, then 
		\[
		\begin{pmatrix}
		1 & 0 \\
		0 & -z
		\end{pmatrix} = \begin{pmatrix}
		1 & 0 \\
		0 & u
		\end{pmatrix}\begin{pmatrix}
		1 & 0 \\
		0 & z
		\end{pmatrix}\begin{pmatrix}
		1 & 0 \\
		0 & u
		\end{pmatrix}.
		\]
		However, if $-1$ is non-square, then $-1 = zc^2$ for some unit $c \in R$ and
		\[
		\begin{pmatrix}
		1 & 0 \\
		0 & c
		\end{pmatrix}\begin{pmatrix}
		1 & 0 \\
		0 & -z
		\end{pmatrix}\begin{pmatrix}
		1 & 0 \\
		0 & c
		\end{pmatrix}
		= \begin{pmatrix}
		1 & 0 \\
		0 & -zc^2
		\end{pmatrix} = I_2.
		\]
		Next, a simple calculation with $P = \frac{1}{2}
		\begin{pmatrix}
		I_\nu & -I_\nu \\
		I_\nu & I_\nu
		\end{pmatrix}$ shows that $L = 2
		\begin{pmatrix}
		I_\nu & 0 \\
		0 & -I_\nu
		\end{pmatrix}$ is cogredient to $H_{2\nu}$. Clearly, if $-1$ is square, $L$ is cogredient to $I_{2\nu}$. Assume that $-1$ is non-square. By Proposition \ref{indexUR} (2), $-1 = zc^2$ for some unit $c$ which also implies that $2$ or $-2$ must be a square unit. If $2$ is a square unit, then 
		\[
		L \approx I_\nu \oplus (-I_\nu) \approx I_\nu \oplus zc^2I_\nu \approx I_\nu \oplus zI_\nu.
		\]
		Similarly, if $-2$ is a square unit, then 
		\[
		L \approx (-I_\nu) \oplus I_\nu \approx zc^2I_\nu \oplus I_\nu \approx I_\nu \oplus zI_\nu.
		\]
		Therefore, $I_\nu \oplus zI_\nu$ is cogredient to $H_{2\nu}$. 
	\end{proof}
\end{lemma}	

Next, we apply Lemmas \ref{zI} and \ref{cH} in the following calculations. We distinguish three cases. Let $z$ be a non-square unit and $\nu$ a positive integer.
\begin{enumerate}
	\item Assume that $-1$ is square. Then
	
	\begin{enumerate}
		\item $I_{2\nu} \approx H_{2\nu}$ and $I_{2\nu + 1} \approx H_{2\nu} \oplus (1)$.
		\item $I_{2\nu} \oplus (z) \approx H_{2\nu} \oplus (z)$ and 
		$ 
		I _{2(\nu - 1)} \oplus (z) \approx I_{2(\nu - 1)} \oplus
		\begin{pmatrix}
		1 & 0 \\
		0 & z
		\end{pmatrix}
		\approx
		H_{2\nu - 1} \oplus 
		\begin{pmatrix}
		1 & 0 \\
		0 & -z
		\end{pmatrix}.
		$
	\end{enumerate}
	\item Assume that $-1$ is non-square and $\nu$ is even. Then
	
	\begin{enumerate}
		\item $I_{2\nu} \approx I_\nu \oplus I_\nu \approx I_\nu \oplus zI_\nu \approx H_{2\nu}$ and 
		$I_{2\nu + 1} \approx I_\nu \oplus I_\nu \oplus (1) \approx I_\nu \oplus zI_\nu \oplus (1) \approx H_{2\nu} \oplus (1)$.
		\item $I_{2\nu} \oplus (z) \approx I_\nu \oplus I_\nu \oplus (z) \approx I_\nu \oplus zI_\nu \oplus (z) \oplus H_{2\nu} \oplus (z)$ and
		\begin{align*}
			I_{2\nu - 1} \oplus (z) 
			&\approx I_{\nu - 2} \oplus I_{\nu - 2} \oplus I_3 \oplus (z)
			\approx I_{\nu - 2} \oplus zI_{\nu - 2} \oplus I_3 \oplus (z)\\
			&\approx I_{\nu - 1} \oplus zI_{\nu - 1} \oplus I_2 
			\approx H_{2(\nu - 1)} \oplus 
			\begin{pmatrix}
				1 & 0 \\
				0 & -z
			\end{pmatrix}.
		\end{align*}
	\end{enumerate}
	\item Assume that $-1$ is non-square and $\nu$ is odd. Then
	
	\begin{enumerate}
		\item  $I_{2\nu} \approx I_{\nu - 1} \oplus I_{\nu - 1} \oplus I_2 \approx I_{\nu - 1} \oplus zI_{\nu - 1} \oplus I_2 \approx H_{2(\nu - 1)} \oplus 
		\begin{pmatrix}
		1 & 0 \\
		0 & -z
		\end{pmatrix}$ and \\
		$I_{2\nu + 1} \approx I_{\nu - 1} \oplus I_{\nu - 1} \oplus I_2 \oplus (1) \approx I_{\nu - 1} \oplus zI_{\nu - 1} \oplus zI_2 \oplus (1) \oplus I_\nu \oplus zI_\nu \oplus (z) \approx H_{2\nu} \oplus (z)$.
		\item $I_{2\nu} \oplus (z) \approx I_{\nu - 1} \oplus I_{\nu - 1} \oplus I_2 \oplus (z) \approx I_{\nu -1} \oplus zI_{\nu - 1} \oplus I_2 \oplus (z) \approx I_\nu \oplus zI_\nu \oplus (1) \approx H_{2\nu} \oplus (1)$ and 
		$I_{2\nu - 1} \oplus (z) \approx I_{\nu - 1} \oplus I_{\nu - 1} \oplus (1) \oplus (z) \approx I_{\nu - 1} \oplus I_{\nu - 1} \oplus (1) \oplus (z) \approx I_\nu \oplus zI_\nu \approx H_{2\nu}$.
	\end{enumerate}
	
\end{enumerate}
This proves a cogredient standard form of an orthogonal space over a finite local ring of odd characteristic.

\begin{thm}\label{cogredientmatrix}
	Let $R$ be a finite local ring of odd characteristic and let $(V, \beta)$ be an orthogonal space where $V$ is  a free $R$-module of rank $n \geq 2$. Then there exists a $\delta \in \{0,1,2\}$ such that $\nu = \dfrac{n-\delta}{2} \ge 1$ and the associating matrix of $\beta$ is cogredient to 
	\[
	S_{2\nu + \delta, \Delta} =
	\begin{pmatrix}
	0 & I_\nu &  \\
	I_\nu & 0 & \\
	&        & \Delta
	\end{pmatrix},
	\]
	where
	\[
	\Delta =
	\begin{cases}
	\emptyset (\text{disappear}) & \text{if} \ \delta = 0, \\
	(1) \ \text{or} \ (z)      & \text{if} \ \delta = 1, \\
	\diag (1, -z)          & \text{if} \ \delta = 2,
	\end{cases} 
	\]
	and $z$ is a fixed non-square unit of $R$. 
\end{thm}

\bigskip

\noindent{\bf Acknowledgments}  I would like to thank the Science Achievement Scholarship of Thailand (SAST) for financial support throughout my undergraduate and graduate study.


\end{document}